\documentclass[preprint,10pt]{elsarticle}

\usepackage{amsmath,amsopn,amssymb,epsfig}%,hyperref}

\newtheorem{Theorem}{Theorem}[section]
\newtheorem{Lemma}{Lemma}[section]
\newtheorem{Proposition}{Proposition}[section]
\newtheorem{Remark}{Remark}[section]

\newtheorem{example}{Example}[section]

\newproof{proof}{Proof}
\newproof{pot}{Proof of Theorem \ref{thm2}}
\newcommand{\bb}{\begin{bmatrix}}
\newcommand{\eb}{\end{bmatrix}}
\newcommand{\bl}[1]{\begin{list}{#1}{\usecounter{bean}}} \newcommand{\el}{\end{list}}
\newcommand{\bel}[1]{\begin{equation} \label{#1}} \newcommand{\eel}{\end{equation}}

\begin{document}

\date{}
\begin{frontmatter}
\title{
On the Sylvester-like matrix equation $AX+f(X)B=C$
\tnoteref{t1}}%,t2}}
\author{Chun-Yueh Chiang\corref{cor1}\fnref{fn1}}
\ead{chiang@nfu.edu.tw}
\address{Center for General Education, National Formosa
University, Huwei 632, Taiwan.}
%\author{Matthew M. Lin \corref{cor1}\fnref{fn2}}
%\ead{mlin@math.ccu.edu.tw}
%\address{Department of Mathematics, National Chung Cheng University, Chia-Yi 621, Taiwan.}

\cortext[cor1]{Corresponding author}
\fntext[fn1]{The author was supported
 by the Ministry of Science and Technology of Taiwan under grant
NSC 102-2115-M-150-002.}
%\fntext[fn2]{The second
%author was supported by the National Science Council of Taiwan under grant
%101-2115-M-194-007-MY3.}

\date{ }

\begin{abstract}
Many applications in applied mathematics and control theory give rise to the unique solution of a Sylvester-like matrix equation associated with an underlying structured matrix operator $f$. In this paper, we will discuss the solvability of the Sylvester-like matrix equation through an auxiliary standard (or generalized) Sylvester matrix equation. We also show that when this Sylvester-like matrix equation is uniquely solvable, the closed-form solutions can be found by using previous result. In addition, with the aid of the Kronecker product some useful results of the solvability of this matrix equation are provided.
\end{abstract}


\end{frontmatter}

\section{Introduction}
Problems in determining the solutions of a matrix equation are closely related to a wide range of
challenging scientific areas. Especially, many linear matrix equations are encountered in many applications of control and engineering problems \cite{Mehrmann91,KwaSiv1972}. There are various results solving many interesting particular cases of this topic. A detailed survey of this area can be found in \cite{Bini2012,Lancaster1995}. Useful results can also be found in \cite{Simoncini14}.
  In this paper, we consider the following  Sylvester-like matrix equation
%\begin{subequations}
\begin{align}\label{eq}
AX+f(X)B=C,
\end{align}
%\end{subequations}
where $A,\,B,\,C \in\mathbb{F}^{m\times m}$ are known matrices, and the $m$-square matrix $X$ is an unknown matrix to be determined. The transformation $f:\mathbb{F}^{m\times m}\rightarrow \mathbb{F}^{m\times m}$ is a
matrix operator enjoying the following suitable properties:
\begin{itemize}
\item[1.]Period-$2$: $f^{(2)}(A):=f(f(A))=A$,
\end{itemize}
and
\begin{itemize}
\item[2a.]Multiplication preserving: $f(AB)=f(A)f(B)$,
\end{itemize}
or
\begin{itemize}
\item[2b.]Multiplication reversing: $f(AB)=f(B)f(A)$,
\end{itemize}
for all $A$ and $B\in\mathbb{F}^{m\times m}$, where the base field $\mathbb{F}$ in these equations
varies between the real field $\mathbb{R}$ and complex field $\mathbb{C}$. Also, we assume that in both cases $f$ is a \emph{general} linear map through the paper, i.e., $f$ satisfies
\begin{itemize}
\item[1.]Additivity: $f(A+B)=f(A)+f(B),\quad A,\,B\in\mathbb{F}^{m\times m}$,
\end{itemize}
and
\begin{itemize}
\item[2.]\emph{General} homogeneous of degree 1:
\begin{align*}
f(a X)=s(a)f(X),\quad a\in\mathbb{F}, X\in\mathbb{F}^{m\times m},
\end{align*}
\end{itemize}
 where $s:\mathbb{F}\rightarrow\mathbb{F}$ is an injective scalar function over $\mathbb{F}$ associated with $f$.

  Assume that $f$  is a bijective linear operator (i.e., $s$ is the identity function). In \cite{emrl2006364}, $f$ is called a \emph{similarity transformation} when $f$ is multiplication reversing, and $f$ is called an \emph{anti-similarity transformation} when $f$ is multiplication preserving. Furthermore, a linear operator $f$ with property (2a) or (2b) implies that $f$ preserves invertibility. See Section~4 for details. %where $I_m$ denoted by the $m\times m$ identity matrix.

Over the past century the following Sylvester-like matrix equations have received considerable attention and have been the topic of many elegant and complete studies:
%\begin{subequations}\label{2}
\begin{itemize}
\item[1.]The standard Sylvester matrix equation
\begin{align}\label{SSylvester}
AX+XB=C,
\end{align}
corresponding to $f$ is the identity operator which is a multiplication preserving operator.

\item[2.]The conjugate Sylvester matrix equation
\begin{align}\label{CSylvester}
AX+\overline{X}B=C,
\end{align}
corresponding to $f$ is the conjugate operator which is a multiplication preserving operator.

\item[3.]The $\top$-Sylvester matrix equation
\begin{align}\label{TSylvester}
AX+X^\top B=C,
\end{align}
corresponding to $f$ is the transpose operator which is a multiplication reversing operator.

\item[4.]The $H$-Sylvester matrix equation
\begin{align}\label{HSylvester}
AX+X^H B=C,
\end{align}
corresponding to $f$ is the conjugate transpose operator which is a multiplication reversing operator.
\end{itemize}
%\end{subequations}
Here, $A,B$, and $C$ are matrices of the appropriate dimensions with entries in the complex or the real field in the corresponding matrix equation. We note that all matrix equations \eqref{SSylvester}--\eqref{HSylvester} are special cases of Sylvester-like matrix equation~\eqref{eq}. In particular, Sylvester-like matrix equations \eqref{SSylvester}--\eqref{HSylvester} arise in applications from a wide range of areas including control theory, signal processing, and partial differential equations. Also, in the field of applied linear algebra, these equations naturally arise in several applied contexts. For instance,
Eq.~\eqref{SSylvester} was first studied by J.J. Sylvester \cite{S84} in 1884 as a tool to reduce block-triangular matrices to block-diagonal form by similarity, and later the solution of Eq.~\eqref{SSylvester} has received a great deal of attention \cite{Horn1994}. In \cite{doi:10.1137/0609029}, Eq.~\eqref{CSylvester} was investigated and conditions for its consistency and solvability were deduced, using the consimilarity canonical form in \cite{Hong1988143}. The main application of Eq.~\eqref{CSylvester} arises from the study of consimilarity \cite{doi:10.1137/0609029,Hong1988143}. For recent surveys on consimilarity and the closely related topic of conjugate-normal matrices, please consult \cite{Fabbender20081425}. Eq.~\eqref{TSylvester} and Eq.~\eqref{HSylvester} are highly related to the problem of reducing block-anti-triangular matrices to block-anti-diagonal matrices by congruence \cite{doi:10.1137/050637601}. These two matrix equations have attracted increasing interest in the past few years, with recent works on the necessary and sufficient conditions for the solvability of Eqs.~\eqref{TSylvester}--\eqref{HSylvester}. See also more recent advanced developments in \cite{Simoncini14}.

 %1210
 It is worthwhile to mention that the Sylvester-likes matrix equations \eqref{SSylvester}--\eqref{HSylvester} have great relevance in control applications associated with the  continuous-time descriptor systems \cite{Mehrmann91,KwaSiv1972} and the vibration of fast trains \cite{Ipsen04,Hilliges2004}. Specifically, when some additional properties or connections are assumed on the coefficient matrices $A$ and $B$ as in Eq.~\eqref{SSylvester}, like being symmetric or Hermitian, the interest focuses on solutions enjoying some of these properties as well. For example, the continuous-time Lyapunov matrix equation can be obtained from the standard Sylvester matrix equation~\eqref{SSylvester} when $B=A^\top$ (or $A^H$). The continuous-time Lyapunov matrix equation has been discovered in many literature, due to the well-known role of this famous matrix equation in control \cite{Mehrmann91}. Another important application in the Sylvester-like matrix equation~\eqref{eq} with a multiplication reversing operator $f$ is related to the study of structured ($f$-palindromic) quadratic eigenvalue problems (we refer this as $f$-palindromic QEP) arising from the simulations of surface acoustic wave filter and the vibration of fast trains \cite{Chu2010,za06}:
\begin{align}\label{QEP}
\mathcal{Q}(\lambda)(x):= (\lambda^2 A_2+\lambda A_1+A_0)x=0,\quad A_{i}=f(A_{2-i}),
\end{align}
where both $A_i$ are $m$-square matrices, for $i=0,1,2$. Interestingly, all eigenvalues of $\mathcal{Q}(\lambda)$ come in $s$-reciprocal pairs $(\lambda,1/s(\lambda))$ if $f$ preserves invertibility since $$\det(\mathcal{Q}(\lambda))=\det(f(s^2(\lambda) A_0+s(\lambda) A_1+A_2)).$$
This result also regards $0$ and $\infty$ as reciprocals of each other. A standard approach for solving the $f$-palindromic QEP is to transform it into a $2m\times2m$ linearized eigenvalue problem. For more detail of this application, see Appendix.
%1210

In this work, we are particularly interested in the theoretical
solutions of the Sylvester-like matrix equation \eqref{eq} and its solvable condition. Following the idea in \cite{Zhou2011}, also a special case in Remark 2 in \cite{Li2014546}, or more general results in \cite{Chiang2014925}, we try to transform Eq.~\eqref{eq} into a standard (or generalized) Sylvester matrix equation. A sufficient condition for existence of the unique solution of the Sylvester-like matrix equation was established by means of elimination by addition or subtraction. Furthermore, we extend this result to the general Sylvester-like matrix equation~\eqref{geq}. The proposed conclusions may provide great convenience to the analysis of such a matrix equation. On the other hand, a theoretical solution of Eq.~\eqref{eq} was derived in explicit form  based on the so-called (relative) Cayley-Hamilton theorem. More precisely, the solution can be expressed in terms of the Laurent expansion of a regular matrix pencil when $f$ is multiplication reversing.
% 0203

Another issue will be addressed in the work which is closely related to the Kronecker product. Indeed, the Kronecker product is an important tool in the analysis and design of numerical algorithms to compute the solutions of Eq.~\eqref{eq}. Moreover, the linear bijective operator $f$ with property (2a) or (2b) can be characterized by a concrete map \cite{emrl2006364}. With the help of the Kronecker product, it give us the complete picture of the inclusion relationship between the solvable condition of  Eq.~\eqref{eq} and the spectral information about the matrix pencil in terms of matrices $A$ and $B$.

Throughout this paper, the symbol $(a_{ij})_{m\times n}$ stands for a $m\times n$ matrix $A$. We denote
  the $m\times m$ identity matrix by $I_m$, the transpose matrix of $A$ by $A^\top$, the conjugate matrix of $A$ by $\overline{A}$, the conjugate transpose matrix of $A$ by $A^H$, and use $\det(A)$ to denote the determinant of a square matrix $A$. The notation $A\otimes B:=(a_{ij}B)_{ms\times nt}$ is denoted by the Kronecker product for a matrix $A=(a_{ij})_{m\times n}$ and a $s\times t$ matrix $B$. Given a matrix pencil $A-\lambda B$ over field $F$, the matrix pencil $A-\lambda B$ is called regular if $\det(A-\lambda B) \neq 0$ for some $\lambda\in \mathbb{F}$.

The paper is organized as follows. We formulate the sufficient conditions for the existence of the solutions of Eq.~\eqref{eq} directly by means of the solvable analysis of the standard (or generalized) Sylvester matrix equation in Section~2. By using the (relative) Cayley-Hamilton theorem, closed-form solutions to a family of Sylvester-like matrix equations are presented in Section 3. A direct method applying the Kronecker product for solving Eq.~\eqref{eq1} is briefly discussed in Section 4. Finally, concluding remarks are given in Section 5.

%Throughout the paper, we use  $A^\top$, $\overline{A}$, $A^H$, and $\det(A)$  to denote the transpose, the conjugate transpose, the conjugate transpose, and the determinant of the matrix $A$, respectively.
%The notations $A\otimes B=[a_{ij}B]$ is denoted by the Kronecker product of matrices $A=[a_{ij}]$ and $B=[b_{ij}]$.
\section{A sufficient condition for the unique solvability of Eq.~\eqref{eq}}
To begin with, we first state the following result in \cite{Chu1987,Herndez1989333}, which is the necessary and sufficient condition for the solvability of the generalized Sylvester matrix equation
%\begin{subequations}
\begin{align}\label{CM}
AXD+CXB=E,
\end{align}
%\end{subequations}
where $A,\,B,\,C,\,D$ and $E$ are $m$-square matrices with entries in the complex or the real field.
\begin {Theorem}\label{Eric0}
Eq.~\eqref{CM} has a unique solution
if and only if two matrix pencils $\lambda C-A$ and $\lambda D-B$ are regular and the spectra of the pencils satisfy  $\sigma(\lambda C-A)\cap\sigma(\lambda D+B)=\phi$.
\end {Theorem}
\begin{Remark}
A special mention should be paid to the case of the standard Sylvester equation~\eqref{SSylvester} whenever
$C$ and $D$ are identity matrices in Eq.~\eqref{CM}, yielding the solvability condition of Eq.~\eqref{SSylvester} will be shown to be the well-known condition
\[
\sigma(A) \cap \sigma(-B)=\phi.
\]
\end{Remark}
A standard way to solve a given matrix equation is to simplify it, by applying favorable
transformations to the unknowns or to the coefficient matrices. In order to transform Eq.~\eqref{eq} into a standard (or generalized) Sylvester matrix equation, the elimination by addition or subtraction will be treated. Taking into account Theorem~\ref{Eric0}, the condition for the existence of the unique solution of Eq.~\eqref{eq} could be explained through some associated generalized Sylvester equations.

Following the idea in [17], the first major step is to link the Eq.~\eqref{eq} to an equivalent
equation if $f$ is multiplication preserving. Applying the matrix operator $f$ to both sides of Eq.~\eqref{eq} we have
 \begin{align}\label{eq1}
 f(A)f(X)+Xf(B)=f(C).
 \end{align}
Pre-multiplying Eq.~\eqref{eq} with $f(A)$ and post-multiplying Eq.~\eqref{eq1} with $-B$, adding two resulting equations to produce the following standard Sylvester matrix equation
\begin{align}\label{SS1}
X(f(B)B)-(f(A)A)X=f(C)B-f(A)C.
\end{align}

Turning now to the case of the multiplication reversing operator $f$, applying matrix operator $f$ to both sides of Eq.~\eqref{eq}
 \begin{align}\label{eq2}
 f(B)X+f(X)f(A)=f(C).
 \end{align}
Again, we should try to get a standard Sylvester equation from Eq.~\eqref{eq} and Eq.~\eqref{eq2}.
However, it is seemingly difficult to eliminate the last term of the left hand side of Eq.~\eqref{eq} or Eq.~\eqref{eq2}. Inspired by the idea of \cite{Zhou09}[Lemma 9], an elegant matrix identity has been introduced as an approach to derive the solvable condition of a specific matrix equation. Given a regular matrix pencil $\mathcal{A}-\lambda\mathcal{B}$, we now define the so-called resolvent matrix operator $Z_\gamma$ such that $Z_\gamma(\mathcal{A},\mathcal{B}):=(\mathcal{A}+\gamma\mathcal{B})^{-1}$ with a suitable parameter $\gamma$ , it is easy to shown that the following result.
\begin{Proposition}\label{swap}
Under the given hypothesis that $\mathcal{A}-\lambda\mathcal{B}$ is regular, then
\[
\mathcal{A}Z_\gamma(\mathcal{A},\mathcal{B})\mathcal{B}=\mathcal{B}Z_\gamma(\mathcal{A},\mathcal{B})\mathcal{A}
\]
with a suitable $\gamma$.
\end{Proposition}
\begin{proof}
The result follows immediately if $\gamma=0$. Otherwise, since
\begin{align*}%\label{swap}
&\gamma\mathcal{B}(\mathcal{A}+\gamma \mathcal{B})^{-1} \mathcal{A}= \mathcal{A} -\mathcal{A}(\mathcal{A}+\gamma\mathcal{B})^{-1} \mathcal{A}
= \mathcal{A}(\mathcal{A}+\gamma\mathcal{B})^{-1}\gamma \mathcal{B},%\label{ZZstar},
\end{align*}
the result also holds in this situation.
\end{proof}
Suppose that $B-\lambda f(A)$ is regular, we choose a scalar $\lambda_0$ such that $B+\lambda_0 f(A)$ is nonsingular.
Multiplying Eq.~\eqref{eq} with $Z_{\lambda_0}(B,f(A))f(A)$  on the right and multiplying the latter equation \eqref{eq2} with $-Z_{\lambda_0}(B,f(A))B$ on the right, adding two resulting equations and applying Proposition~\ref{swap}  yields that the generalized Sylvester matrix equation
\begin{align}\label{SS2}
A{\bf X} Z_{\lambda_0} (B,f(A))f(A)-f(B){\bf X}Z_{\lambda_0} (B,f(A))B
=E,
\end{align}% Note that $B-\lambda_0 f(A)$ is nonsingular implies that $f(B)-\lambda_0 A$ is invertible.
where $E=CZ_{\lambda_0} (B,f(A))f(A)-f(C)Z_{\lambda_0} (B,f(A))B$. Subsequently, from Eq.~\eqref{SS1}, Eq.~\eqref{SS2} and Theorem~\ref{Eric0} the main result in this section can be summarized in the following theory:
\begin{Theorem}\label{unique}
A sufficient condition for the unique solvability of  Eq.~\eqref{eq} is
\begin{itemize}
\item[a.] $\sigma(Af(A)) \cap \sigma(Bf(B)) =\phi$ if $f$ is multiplication preserving,
\end{itemize}
and
\begin{itemize}
\item[b.] $\sigma(A-\lambda f(B)) \cap \sigma(B-\lambda f(A)) =\phi$ for two regular matrix pencils $A-\lambda f(B)$ and $B-\lambda f(A)$ if $f$ is multiplication reversing.
\end{itemize}
\end{Theorem}
Theorem~\ref{unique} provides a sufficient condition on the existence of the unique solution to Eq.~\eqref{eq}.
It is interesting to ask whether this sufficient condition is also the necessary condition for the uniqueness of solution of Eq.~\eqref{eq}. Unfortunately, it is disappointed with the answer from the following counterexample.
\begin{example}
Consider the scalar equation ($m=1$) in the form of Eq.~\eqref{eq} with
\begin{align*}
ax+f(x)b=c, \quad a,b,c\in\mathbb{R}.
\end{align*}
Let $f$ be the identity map and $a=b=1$. Of course, $f$ is also the transpose operator. It is clear that the scalar equation $x+x=c $ has a unique solution $x=c/2$. However, both conditions (a) and (b) in Theorem~\ref{unique} are not satisfied.
\end{example}
%0128
Clearly, any solution of Eq.~(1) is also a solution of Eq.~\eqref{SS1} (or Eq.~\eqref{SS2}), but the converse statement is not true. It seems that we need additional information about $f$ to determine the necessary condition on the existence of the unique solution to Eq.~\eqref{eq}.

It is interesting and exciting that the different property of $f$ makes the
equation behave very differently. The solvable condition in terms of non-intersection of the spectra $\sigma(A)$ and $\sigma(B)$, for the standard Sylvester equation~\eqref{SSylvester}, is shifted to condition (b) in Theorem~\ref{unique} for the generalized spectrum $\sigma(A,B^\top)$. In
addition, $\top$-Sylvester matrix equation~\eqref{TSylvester} looks like a Sylvester equation associated with continuous-time but condition (b) in Theorem~\ref{unique} is satisfied when $\sigma(A,B^\top)$ in totally inside the unit circle, hinting at a discrete-time type of stability behavior.
\begin{Remark}
\noindent\begin{itemize}
\item[1.]
If $f$ preserves invertibility, $\lambda\in\sigma(A-\lambda f(B))$ implies that $1/s(\lambda)\in\sigma(B-\lambda f(A))$ since $\det(B-\lambda f(A))=\det(f(f(B)-s(\lambda)A))$. Then, the condition (b) in Theorem~\ref{unique}  can be rewritten as the so-called ``$s$-reciprocal free'': The $s$-reciprocal pair $(\lambda,1/s(\lambda))$ cannot belong to $\sigma(A-\lambda f(B))$, this definition also regards $0$ and $\infty$ as reciprocals of each other.
\item[2.]
The sufficient condition for the existence of the unique solution of the following generalized Sylvester matrix equation
\begin{align}\label{geq}
AXD+Ef(X)B=C,
\end{align}
can be obtained in the similar manner as mentioned before. We state the result in the following without the proof.
\begin{Theorem}\label{gunique}
 %For the sake of simplicity we define two matrix operators $U_{\gamma}(A,B):=AZ_{\gamma}(A,B)$ and $V_{\gamma}(A,B):=Z_{\gamma}(A,B)A$ with matrices $A$ and $B$, and a scalar $\gamma$.
 With two suitable parameters $\gamma_1$ and $\gamma_2$, a sufficient condition for the unique solvability of  Eq.~\eqref{geq} is
\begin{itemize}
\item[a.] $\sigma(f(A) Z_{\gamma_1}(E,f(A))A-\lambda E Z_{\gamma_1}(E,f(A))f(E)) \cap\, \sigma(f(B) Z_{\gamma_2}(B,f(D))B-\lambda D Z_{\gamma_2}(B,f(D))f(D)) =\phi$ when $f$ is multiplication preserving. Moreover, the condition can be reduced to
     \[
     \sigma(f(A)A-\lambda Ef(E)) \cap \sigma(f(B)B-\lambda Df(D))=\phi
     \]
     if $E$ and $f(A)$ are commuting,  $B$ and $f(D)$ are commuting,
\end{itemize}
and
\begin{itemize}
\item[b.] $\sigma(f(D) Z_{\gamma_1}(E,f(D))A-\lambda E Z_{\gamma_1}(E,f(D))f(B)) \cap\, \sigma(f(E) Z_{\gamma_2}(B,f(A))B-\lambda D Z_{\gamma_2}(B,f(A))f(A)) =\phi$ when $f$ is multiplication reversing.
     Moreover, the condition can be reduced to
     \[
     \sigma(f(D)A-\lambda Ef(B)) \cap \sigma(f(E)B-\lambda Df(A))=\phi
     \]
     if $E$ and $f(D)$ are commuting,  $B$ and $f(A)$ are commuting.
 \end{itemize}
All matrix pencils mentioned above are assumed to be regular.
\end{Theorem}
\end{itemize}
\end{Remark}
\section{The closed-form solutions of Eq.~\eqref{eq}}
In the past few decade, Eqs.~\eqref{SSylvester}--\eqref{CSylvester} have the unique solution under certain conditions in terms of the coefficient matrices $A$ and $B$, with available elegant and explicit closed-forms \cite{Jameson1968,6595170}. So far, the closed-form solution of Eq.~\eqref{eq} has been established if $f$ is multiplication preserving, but has not been explored when $f$ is multiplication reversing.  In each case, we will present a closed-form solution of Eq.~\eqref{eq} in this section. This closed-form solution is expressed in terms of the coefficient matrix of Eq.~\eqref{eq}.

If $f$ is multiplication preserving, Jameson \cite{Jameson1968} provides an explicit solution of Eq.~\eqref{SSylvester}
by the approach of characteristic polynomial of matrix $A$ (or matrix $B$). Based on the assumption in (a) of Theorem~\ref{unique} , together with Eq.~\eqref{eq1} the following result is an immediate consequence of \cite{Jameson1968}.
\begin{Theorem}\label{explicitthm}
Assume that the sufficient condition (a) in Theorem~\ref{unique} holds. Let $\mathcal{A}:=f(A)A$, $\mathcal{B}:=f(B)B$ and $\mathcal{C}:=f(C)B-f(A)C$ in Eq.~\eqref{SS1}. We denote the characteristic polynomial of $\mathcal{A}$ by $\mbox{ch}_{\mathcal{A}}(\lambda):=\det(\lambda I_m-\mathcal{A})=\sum\limits_{k=0}^m p_k \lambda^k$. Then, the matrix $\mbox{ch}_{\mathcal{A}}(\mathcal{B})$  is nonsingular. Moreover, the unique solution $X$ of Eq.~\eqref{eq} has the following explicit form,
\begin{align*}\
X=\left(\sum\limits_{i=1}^{m}\sum\limits_{k=1}^{i-1} p_i \mathcal{A}^{k} \mathcal{C} \mathcal{B}^{i-k-1}\right)\mbox{ch}_{\mathcal{A}}(\mathcal{B})^{-1}.
\end{align*}
Alternatively, we express the characteristic polynomial of $\mathcal{B}$ as $\mbox{ch}_{\mathcal{B}}(\lambda)=\sum\limits_{k=0}^m q_k \lambda^k$. Then, the matrix $\mbox{ch}_{\mathcal{B}}(\mathcal{A})$ is also nonsingular and $X$ can be presented by the following  explicit solution,
\begin{align*}\
X=-\mbox{ch}_{\mathcal{B}}(\mathcal{A})^{-1}\left(\sum\limits_{i=1}^{m}\sum\limits_{k=1}^{i-1} q_i \mathcal{A}^{k} \mathcal{C} \mathcal{B}^{i-k-1}\right).
\end{align*}
\end{Theorem}
Now we should pay attention to finding an explicit expression of the solution $X$ of Eq.~\eqref{eq} when $f$ is multiplication reversing. The original idea is analogous to the derivations of an explicit solution of the
generalized Sylvester matrix equation~\eqref{CM} in \cite{Herndez1989333}. To obtain the result, the following preliminary result is needed. The following theorem in \cite{Langenhop1971329,Schweitzer1993237} gives a special expansion of a regular matrix pair $A-\lambda B$ with respect to $\lambda$, which plays a crucial role in this section.
\begin{Theorem}\label{Laurent}
Given two $m \times m$ matrices $D$ and $E$. If the matrix pencil
$D-\lambda E$ is regular, then there exists a nonnegative integer $\mu$ and $m\times m$ matrices $U_k$ such that
\begin{align}\label{LaurentE}
(D-\lambda E)^{-1}=\frac{1}{\lambda}\sum\limits_{k=-\mu}^\infty U_k \lambda^{-k},
\end{align}
for all scalar $\lambda$ in a deleted neighborhood of zero $B_\delta (0)\backslash \{0\}$, where $\delta>0$.
\end{Theorem}
The analytic expression of $(D-\lambda E)^{-1}$ in \eqref{LaurentE} is called the Laurent expansion of $(D-\lambda E)^{-1}$ about $\lambda=0$, see \cite{Rothblum198133} for the further work. We denote the so-called ``relative characteristic polynomial'' of the pencil $D-\lambda E$ by $\mbox{ch}_{D,E}(\lambda):=\det(D-\lambda E)$. The following result can be regard as the generalization of the Cayley-Hamilton theorem; see, e.g., \cite{4047765}.
\begin{Lemma}\label{Coro}
If we write the expression of the relative characteristic polynomial of the $m\times m$ matrix pencil $D-\lambda E$ in ascending power in $\lambda$ as $\mbox{ch}_{D,E}(\lambda)=\sum\limits_{j=0}^m p_j \lambda^j$, then
\[
\mbox{ch}_{D,E}(U_k):=\sum\limits_{j=0}^m p_j U_{k+j-m}= 0,\quad  \mbox{for } k\geq m \mbox{ or } k\leq -1
\]
, where the sequence of matrices $\{U_j\}$ is defined in the infinite matrix series  \eqref{LaurentE} of Theorem~\ref{Laurent}.
\end{Lemma}
With the notation $\mbox{ch}_{D,E}(U_k)$ in Lemma~\ref{Coro}, we have the following adaptation of \cite{Herndez1989333}[Theorem 1].
\begin{Lemma}\label{Chiang}
Given two regular $m$-square matrix pencils $D-\lambda E$ and $F-\lambda G$. Let the Laurent expansion of $(F-\lambda G)^{-1}$ be $(F-\lambda G)^{-1}=\frac{1}{\lambda}\sum\limits_{k=-\nu}^\infty V_k \lambda^{-k}$ with a nonnegative integer $\nu$. If the intersection of the spectra $\sigma(D-\lambda E)$ and $\sigma(F-\lambda G)$ is empty and the matrix $G$ is nonsingular, then $\mbox{ch}_{D,E}(V_m)$ is nonsingular.
\end{Lemma}
\begin{proof}
With the help of the generalized Schur decomposition, we may assume without loss of generality that both $D-\lambda E=(d_{ij}-\lambda e_{ij})_{m\times m}$ and $F-\lambda G=(f_{ij}-\lambda g_{ij})_{m\times m}$ are upper triangular matrix pencils.
%Fix $\lambda\in\mathbb{F}$ such that the matrix$K_\lambda:=(F-\lambda G)^{-1}$ exist.
%We have the following observations.
The proof can be given explicitly as follows.
\begin{itemize}
\item[1.]The inverse of an invertible upper triangular matrix $F-\lambda G$ is also upper triangular. The $i$th diagonal element of $Z_\lambda(F,-G)$ is ${1}/({f_{ii}-\lambda g_{ii}})$.
\item[2.]The Laurent expansion of the scalar function ${1}/({f_{ii}-\lambda g_{ii}})$ with respect to $\lambda$ is $\frac{1}{\lambda}\sum\limits_{k=-1}^{\infty} \alpha_k^{(i)}\lambda^{-k}$, where
    \[
    \alpha_k^{(i)}=\left\{
    \begin{array}{cc}
      -{f_{ii}^{k}}/{g_{ii}^{k+1}}, & g_{ii}\neq 0,\,k>-1,  \\
      1/f_{ii}, & g_{ii}=0,\,k=-1,\\
      0, & \mbox{otherwise}.
    \end{array}
    \right.
    \]
  We conclude that $V_k$ is an upper triangular matrix and the $i$th diagonal element of $V_k$ coincides with $\alpha_k^{(i)}$ by comparing terms of both sides of the Laurent expansion of $(F-\lambda G)^{-1}$.
  % \[
%    (V_k)_{ii}=\left\{
%    \begin{array}{cc}
%      \alpha_k, & g_{ii}\neq 0,  \\
%      1/f_{ii}, & g_{ii}=0,\,k=-1,\\
%      0, & g_{ii}=0,\,k>-1.
%    \end{array}
%    \right.
%    \]

 \item[3.]The $i$th diagonal element of $\mbox{ch}_{D,E}(V_k)$ is $\Delta_{k,i}:=\sum\limits_{j=0}^m p_j (V_{k+j-m})_{ii}$,
  which is equal to
\[
    \Delta_{k,i}=\left\{
    \begin{array}{cc}
      -{f_{ii}^{k-m}}/{g_{ii}^{k-m+1}}\prod\limits_{j=1}^m (d_{jj}-e_{jj}f_{jj}/g_{jj}), & g_{ii}\neq 0,\,k>-1,  \\
       p_j/f_{ii}, &  g_{ii}=0,\,k=-1+m-j,\\
       0, & \mbox{otherwise.}
    \end{array}
    \right.
\]

\end{itemize}
   When $k=m$, $\det(\mbox{ch}_{D,E}(V_m))=\prod\limits_{i=1}^m\Delta_{m,i}\neq 0$ follows from the assumptions that $d_{ii} g_{ii}\neq e_{ii} f_{ii}$, and  $g_{ii}\neq 0$, $1\leq i\leq m$. We complete the proof.
\end{proof}
Armed with the properties given in Theorem~\ref{Laurent}, Lemma~\ref{Coro} and Lemma~\ref{Chiang}, we are ready to derive a closed-form solution of Eq.~\eqref{eq} if $f$ is multiplication reversing. Applying the matrix operator $f$ to both sides of Eq.~\eqref{eq}, multiplying the new equation by $\lambda$, and subtracting the resulting equation from Eq.~\eqref{eq} we obtain
\begin{align}\label{t1}
&X(B-\lambda f(A))^{-1}+(A-\lambda f(B))^{-1}f(X)\nonumber\\
&=(A-\lambda f(B))^{-1}(C-\lambda f(C))(B-\lambda f(A))^{-1}.
\end{align}
 According to Theorem~\ref{Laurent} and the assumption that the regularity of two matrix pencils $A-\lambda f(B)$ and $B-\lambda f(A)$, there exist two sequences of matrices $\{ U_k\}$ and $\{ V_k\}$ such that
%\begin{subequations}
\begin{align}\label{Laurent2}
(A-\lambda f(B))^{-1}=\frac{1}{\lambda}\sum\limits_{k=-\mu_1}^\infty U_k \lambda^{-k},\quad
(B-\lambda f(A))^{-1}=\frac{1}{\lambda}\sum\limits_{k=-\mu_2}^\infty V_k \lambda^{-k}
\end{align}
%\end{subequations}
for $\lambda$ in a deleted neighborhood of zero, where $\mu_1$ and $\mu_2$ are two nonnegative integers.
For the sake of derivation, we without loss of generality assume that $\mu_1$ is grater than or equal to  $\mu_2$. %let $\mu=\max\{\mu_1,\mu_2\}$,
Substituting \eqref{Laurent2} into \eqref{t1} we have
\begin{align}\label{compar}
\frac{1}{\lambda}\sum\limits_{k=-\mu_1}^\infty (XU_k+V_k f(X)) \lambda^{-k}=\frac{1}{\lambda}\sum\limits_{k=-2\mu_1}^\infty T_k \lambda^{-k},
\end{align}
where $T_k=\sum\limits_{\tiny\begin{array}{c}s+t=k\\ s,t\geq-\mu_1\end{array}}U_s C V_t-\sum\limits_{\tiny\begin{array}{c}s+t=k-1\\ s,t\geq-\mu_1\end{array}}U_s f(C) V_t$, and $V_k$ vanishes if $k<-\mu_2$. Comparing both sides in \eqref{compar} we get
\begin{align}\label{compar1}
XU_j+V_j f(X)=T_j,\quad j\geq -\mu_1.
\end{align}
 We expand the relative characteristic polynomial of the pencil $B-\lambda f(A)$ as $\mbox{ch}_{B,f(A)}(\lambda)=\sum\limits_{j=0}^m p_j \lambda^j$
. For $j=0,\cdots,m$, multiplying both sides in \eqref{compar1} by $p_{j}$ and adding these resulting equations can lead to
\begin{align}\label{compar2}
X\mbox{ch}_{B,f(A)}(U_m)+\mbox{ch}_{B,f(A)}(V_m) f(X)=\sum\limits_{j=0}^m p_j T_{j}.
\end{align}
%by applying Theorem~\ref{Coro}.
  %Based on the regularity of the matrix pencil $A-\lambda f(B)$, and assume that the operator $f$ preserves invertibility. We note that either $A$ or $B$ is invertible if the uniquely solvable condition (b) is satisfied. %We may assume without loss of generality that $A$ is nonsingular.
  Assume that the operator $f$ preserves invertibility and the uniquely solvable condition (b) in Theorem~\ref{unique} holds. We note that either $A$ or $B$ is invertible based on the regularity of the matrix pencil $A-\lambda f(B)$. Finally, the main results can be constructed, as combined Lemma~\ref{Coro} with Lemma~\ref{Chiang}.
\begin{Theorem}\label{Chiang Main}
Suppose that the operator $f$ preserves invertibility. Under the uniquely solvable condition (b) in Theorem~\ref{unique} and the assumption that $A$ is nonsingular. The explicit solution of $X$ can be
formulated as $$X=\sum\limits_{j=0}^m p_j T_{j} (\mbox{ch}_{B,f(A)}(U_m))^{-1}.$$
\end{Theorem}
Furthermore, the alternative result of Theorem~\ref{Chiang Main} can be established if we consider the
equivalent equation~\eqref{eq2}. The detail is omitted here.
\section{Some further remarks on Eq.~\eqref{eq}}
Despite a lot of intense work in the past \cite{Chiang2014925,Chiang2012,Chiang2013AAA}, some observations and minor results are provided in this section. Generally speaking, the Kronecker product is a useful and powerful tool for finding the theoretical solution of systems of linear matrix equations \cite{Horn1994}. Especially, some closed-form solutions of a family of generalized Sylvester matrix equation are given by using a specific Kronecker matrix polynomials with the form $\sum\limits_{i=1}^n A_i^\top \otimes B_i$ \cite{Zhou2009327}. In the following, we sketch a big picture on how the Kronecker product has been solved Eq.~\eqref{eq}. With the help of the concept of the Kronecker product, Eq.~\eqref{eq} can be written as
\begin{align}\label{Kron}
(I_m\otimes A)\mbox{vec}(X)+(B^\top\otimes I_m)\mbox{vec}(f(X))=\mbox{vec}(C),
\end{align}
where $\mbox{vec}(X)$ stacks the columns of any $m$-square matrix $X$ onto a column vector. Assume that $f$ is a linear operator, it is easy to see that the composition of three maps $\mbox{vec}\circ f\circ \mbox{vec}^{-1}$ is also a linear transformation from $\mathbb{F}^{m^2\times 1}$ to itself. Then it is known that such a composition mapping $\mbox{vec}\circ f$ can be represented by a matrix $\mathcal{K}_f$ of size $m^2\times m^2$. That is,
\begin{align}\label{cond1}
\mbox{vec}(f(X))=\mathcal{K}_f\mbox{vec}(X),\quad  X\in\mathbb{F}^{m\times m}.
\end{align}
It is obviously that $\mathcal{K}_f$ is an identity matrix of size $2m$ if $f$ is an identity operator.
If $f$ is the transpose operator,  let us now introduce the commutation matrix \cite{1979} (or Kronecker permutation matrix) as follows
\begin{equation*}
\mathcal{K}_f=
\sum\limits_{1\leq i,j\leq m}e_je_i^\top \otimes e_ie_j^\top,
\end{equation*}
where $e_i$ denotes the $i$th column of the $m\times m$ identity
matrix $I_{m}$. As is well-known, its central  property is that it transforms $\mbox{vec}(X)$ into $\mbox{vec}(X^\top)$. However, Eq.~\eqref{cond1} is a contradiction when $f$ is a conjugate(-transpose) operator. To see this, for $m=1$ we choose a scalar $k\in\mathbb{C}$, it is impossible that $\overline{x}=kx$ for all $x\in\mathbb{C}$. The reason is that the conjugate(-transpose) operator is not a linear transformation over field $x\in\mathbb{C}$ since it is not homogeneous of degree 1.

The connection between linear operator $f$ and matrix representation is apparent from \eqref{cond1}.
Surprisingly, $\check{\mbox{S}}$emrl recently characterized the important property of a multiplication preserving/reversing operator, which was shown in \cite{emrl2006364} to
provide a connection between the geometry of matrices and the
structural results for order preserving maps.
\begin{Theorem}\label{emrl}
Suppose that $f$ is a bijective linear operator over $\mathbb{F}$, there exists a nonsingular $T_f$ such that
\begin{itemize}
\item[a.] $f(X)=T_fXT_f^{-1}$ if $f$ is multiplication preserving,
\end{itemize}
or
\begin{itemize}
\item[b.] $f(X)=T_fX^\top T_f^{-1}$ if $f$ is multiplication reserving,
\end{itemize}
for all $X\in\mathbb{F}^{m \times m}$.
\end{Theorem}
According to the Theorem~\ref{emrl} we focus on the identity operator and transpose operator if $f$ is a linear operator.  The solvability of original Eq.~\eqref{eq} is equivalent to the solvability of the generalized Sylvester-like matrix equation~\eqref{geq} with the identity or transpose operator $f$. In terms of the (generalized and periodic) Schur decomposition, QR decomposition and (generalized) singular value decomposition, Eq.~\eqref{geq} can be converted into an upper (or lower) triangular system~\cite{Chiang2012}. With the Kronecker product approach coupled with matrix representation \eqref{cond1}, the following result carries the full spectral information about the matrices in the left hand side of Eq.~\eqref{eq} as well as Eq.~\eqref{Kron}.
\begin{Proposition}\label{eig}
Given four $m$-square upper (or lower) triangular matrices $A$, $B$, $C$ and $D$ and define the $m^2$-square matrix
$\mathcal{P}:=A\otimes B+(C\otimes D)\mathcal{K}_f$, then
\begin{itemize}
\item[a.] $\sigma(\mathcal{P})=\{a_{ii}b_{jj}+c_{ii}d_{jj};1\leq i,j \leq m\}$ if $f$ is an identity operator,
\end{itemize}
and
\begin{itemize}
\item[b.] $\sigma(\mathcal{P})=\{a_{ii}b_{ii}+c_{ii}d_{ii};1\leq i \leq m\}\cup \{\sigma(\bb a_{ii}b_{jj} & c_{ii}d_{jj}\\ c_{jj}d_{ii} &a_{jj}b_{ii} \eb);1\leq i,j \leq m,i\neq j\}$ if $f$ is a transpose operator.
\end{itemize}
\end{Proposition}
\begin{proof}
Part (a) immediately follows from the definition of the Kronecker product. The second conclusion follows the similar arguments as in \cite{Chiang2013AAA}[Lemma3] (the tedious details are omitted).
\end{proof}
\begin{example}\label{ex3}
Let $\sigma_m$ be the set which collects all permutations of the set $M=\{1,\cdots,m\}$. Further, let $\sigma_m^{(k)}$ be the $k$th element of $\sigma_m$ according to Lexicographical order and $P_k\in\mathbb{R}^{m \times m}$ be a permutation matrix whose $j$th row is $e_{\sigma^{(k)}_{m}}^\top$, where $1\leq k \leq m!$. For instance, $P_7=\bb 0 & 1 & 0 & 0\\ 1 & 0 & 0 & 0 \\ 0 & 0 & 1 & 0\\ 0 & 0 & 0 & 1 \eb$ with $m=4$.
Now, we consider a family of linear matrix operators $S=\{ f_k\}$, where $f_k:\mathbb{R}^{m\times m}\rightarrow\mathbb{R}^{m\times m}$ %permutes entries of $m\times m$ matrix $A$.
 is defined by $f_k(A)$ whose entries are permutation of entries of a given $m\times m$ matrix $A$.
 That is, there is a bijective function $\hat\sigma_k:M\times M\rightarrow M\times M$ such that $f_k(A)=(a_{\hat\sigma_k(i,j)})_{m\times m}$. Obviously the total number of $S$ is $(m^2)!$ and there are two subsets $S_1=\{f_k^{(P)}\}$ and $S_2=\{f_k^{(R)}\}$ of $S$ such that
\begin{align}\label{operator}
f_k^{(P)}(X)=P_k X P_k^\top\quad\mbox{and}\quad f_k^{(R)}(X)=P_k X^\top P_k^\top,
\end{align}
where $1\leq k \leq m!$. It follows from Theorem~\ref{emrl} that the set of all multiplication preserving (or reserving) operators in $S$ is exactly $S_1$(or $S_2$). Together with Proposition~\ref{eig}, we summarize the uniquely solvable condition of $f_k^{(P)}$ and $f_k^{(R)}$ in the following theorem:
\begin{Theorem}
The Sylvester-like matrix equation~\eqref{eq} associated with matrix operator $f_k^{(P)}$ or $f_k^{(R)}$ is uniquely solvable if and only if the following conditions hold,
\begin{itemize}
 \item[a.] $\sigma(P_k^\top A) \cap \sigma(P_k^\top B) =\phi$ if the matrix operator $f=f_k^{(P)}$ in \eqref{operator},
\end{itemize}
 and
 \begin{itemize}
  \item[b.]
 \begin{itemize}
\item[(1)]  $\lambda_1,\lambda_2\in\sigma(P_k^\top A-\lambda B^\top P_k)\backslash\{1\}$ implies that $\lambda_1\lambda_2\neq 1$, this definition also regards $0$ and $\infty$ as reciprocals of each other.
\item[(2)] $1$ can be an eigenvalue of the matrix pencil $P_k^\top A-\lambda B^\top P_k$, but must be simple.
\end{itemize}
if the matrix operator $f=f_k^{(R)}$ in \eqref{operator}, $1\leq k \leq m!$.
\end{itemize}
\end{Theorem}
\end{example}
\section{Concluding Remark}
% 1210
The paper studies a class of Sylvester-like matrix equations.
%, for a vast array of applications in palindromic eigenvalue problem, optimal control, etc.
 We present some useful sufficient conditions for the solvability of this linear matrix equation. In addition,  we give the expressions of the explicit solutions of the equations when their solvable conditions are satisfied. A closed-form solution is expressed in terms of the expansion of Laurent series of a regular matrix pencil. The closed-form solutions allow us to calculate the theoretical solution of this linear matrix equation.
 Essentially, a linear operator $f$ with property (2a) can be viewed as the identity operator and a linear operator  $f$ with property (2b) can be regarded as the transpose operator. As mentioned before, the standard Sylvester matrix equation and the transpose Sylvester matrix equation have a special interest in control theory due to the important role in linear continuous-time system theory. Such a mark may bring much advantage in some theoretical analysis and designing numerical algorithms related to the Sylvester-like matrix equation.

\section*{Appendix : $f$-Palindromic Linearization $\lambda F(\mathcal{Z})+\mathcal{Z}$}%\footnote{***FOR NCTS EYES ONLY}}
First, let $M$ be a $2m$-square matrix partitioned as $M=\bb M_1 & M_2\\M_3 & M_4\eb$. The auxiliary matrix operator $F:\mathbb{F}^{m\times m}\rightarrow \mathbb{F}^{m\times m}$ associated with a multiplication reversing operator $f$  is defined as follows:
 $$F(M)=\bb f(M_1) & f(M_3)\\f(M_2) & f(M_4)\eb.$$
It is easy to see that $F$ is a multiplication reversing operator. Furthermore, QEP~\eqref{QEP} can be solved using various linearizations, such as the following
$f$-palindromic linearization of the form
\begin{align}\label{LEP}
\lambda F(\mathcal{Z})+\mathcal{Z},
\end{align}
with $2m$-square matrix $\mathcal{Z}$, where $\mathcal{Z}$ has the partitioned form
\begin{align*}
\mathcal{Z}=\bb A_0 & -A_1-A_2\\ A_0 & A_0 \eb.
\end{align*}
A direct computation yields
\[
(\lambda F(\mathcal{Z})+\mathcal{Z})\bb x\\-\lambda x \eb=0.
\]
We rewrite $\mathcal{Z}$ as $\mathcal{Z}=\bb A & B\\ C & D \eb$ and $\mathcal{X}:=\bb I_m & 0 \\ X & I_m \eb$. By  applying the ``$F$-congruence'' to the ``$F$-palindromic matrix pencil'' \eqref{LEP} we have
\[ \mathcal{X} (\lambda F(\mathcal{Z})+ \mathcal{Z}) F(\mathcal{X}) = \]
\[ \bb \lambda A + f(A) & \lambda (A f(X) + B) + f(XA + C) \\ \lambda (XA + C) + f(A f(X) + B) & \lambda \mathcal{R}(X) + f(\mathcal{R}(X)) \eb \]
with the so-called $f$-Riccati matrix equation,
\[ \mathcal{R}(X) \equiv XAf(X) + XB + Cf(X) + D \ . \]
If we can solve this generalized quadratic matrix equation
\[ \mathcal{R}(X) = 0 \ , \]
then the $f$-palindromic linearization can be ``square-rooted''. We thus have to solve the generalized eigenvalue problem for the pencil $\lambda (A f(X) + B) + f(XA + C)$, with the reciprocal eigenvalues in $\lambda (XA + C) + f(A f(X) + B)$ obtained for free. The Newton's method for the solution of the $f$-Riccati equation leads to the iterative process.
%When the $f$-Riccati equation is applied by Newton's method, which lead to the iterative process
\[
\Delta X_{k+1} (Af(X_k)+B)+(X_k+C) f(\Delta X_{k+1})=-\mathcal{R}(X_k),\quad \Delta X_{k+1}:=X_{k+1}-X_k,
\]
then the Sylvester-like matrix equation~\eqref{eq} need to be solved in each iteration.

\section*{Acknowledgement}
The author wishes to thank the Editor and two anonymous referees for many interesting and valuable suggestions on the manuscript. This research work was partially supported by the Ministry of Science and Technology and the National Center for Theoretical Sciences in Taiwan. The author would like to thank the support from Ministry of Science and Technology under the grant number NSC 102-2115-M-150-002.

%
%%
%\bibliographystyle{elsarticle-num}
%\bibliography{LinearEq}

\begin{thebibliography}{10}
\expandafter\ifx\csname url\endcsname\relax
  \def\url#1{\texttt{#1}}\fi
\expandafter\ifx\csname urlprefix\endcsname\relax\def\urlprefix{URL }\fi
\expandafter\ifx\csname href\endcsname\relax
  \def\href#1#2{#2} \def\path#1{#1}\fi

\bibitem{Mehrmann91}
V.~Mehrmann, The Autonomous Linear Quadratic Control Problem, Springer-Verlag, Berlin, 1991.

\bibitem{KwaSiv1972}
H.~Kwakernaak, R.~Sivan, Linear Optimal Control Systems, Wiley-Interscience,
  New York, 1972.

\bibitem{Bini2012}
D.~A. Bini, B.~Iannazzo, B.~Meini, Numerical Solution of Algebraic {R}iccati
  Equations, Vol.~9 of Fundamentals of Algorithms, Society for Industrial and
  Applied Mathematics (SIAM), Philadelphia, PA, 2012.

\bibitem{Lancaster1995}
P.~Lancaster, L.~Rodman, Algebraic Riccati Equations, Clarendon Press, Oxford
  University Press, New York, 1995.

\bibitem{Simoncini14}
V.~Simoncini, Computational methods for linear matrix equations (Survey article), Technical Report,
  Dipartimento di Matematica, Universita di Bologna, Italy (2014).

\bibitem{emrl2006364}
P.~$\check{\mbox{S}}$emrl,
  \href{http://www.sciencedirect.com/science/article/pii/S0024379505001588}{Maps
  on matrix spaces}, Linear Algebra and its Applications 413~(2–3) (2006) 364
  -- 393, Special Issue on the 11th Conference of the International Linear
  Algebra Society, Coimbra, 2004 11th Conference of the International Linear
  Algebra Society, Coimbra, 2004.
\newblock \href {http://dx.doi.org/http://dx.doi.org/10.1016/j.laa.2005.03.011}
  {\path{doi:http://dx.doi.org/10.1016/j.laa.2005.03.011}}.
\newline\urlprefix\url{http://www.sciencedirect.com/science/article/pii/S0024379505001588}

\bibitem{S84}
J.~J. Sylvester, Sur l'equations en matrices $px = xq$, Comptes Rendus Acad.
  Sci. Paris 99~(2) (1884) 67--71,115--116.

\bibitem{Horn1994}
R.~A. Horn, C.~R. Johnson, Topics in Matrix Analysis, Cambridge University
  Press, Cambridge, 1994, Corrected reprint of the 1991 original.

\bibitem{doi:10.1137/0609029}
J.~Bevis, F.~Hall, R.~Hartwig, \href{http://dx.doi.org/10.1137/0609029}{The
  matrix equation {$A\overline{X} - XB = C$} and its special cases}, SIAM
  Journal on Matrix Analysis and Applications 9~(3) (1988) 348--359.
\newblock \href {http://arxiv.org/abs/http://dx.doi.org/10.1137/0609029}
  {\path{arXiv:http://dx.doi.org/10.1137/0609029}}, \href
  {http://dx.doi.org/10.1137/0609029} {\path{doi:10.1137/0609029}}.
\newline\urlprefix\url{http://dx.doi.org/10.1137/0609029}

\bibitem{Hong1988143}
Y.~Hong, R.~A. Horn,
  \href{http://www.sciencedirect.com/science/article/pii/0024379588903242}{A
  canonical form for matrices under consimilarity}, Linear Algebra and its
  Applications 102~(0) (1988) 143 -- 168.
\newblock \href
  {http://dx.doi.org/http://dx.doi.org/10.1016/0024-3795(88)90324-2}
  {\path{doi:http://dx.doi.org/10.1016/0024-3795(88)90324-2}}.
\newline\urlprefix\url{http://www.sciencedirect.com/science/article/pii/0024379588903242}

\bibitem{Fabbender20081425}
H.~Fa{\ss}bender, K.~Ikramov,
  \href{http://www.sciencedirect.com/science/article/pii/S0024379508001456}{Conjugate-normal
  matrices: A survey}, Linear Algebra and its Applications 429~(7) (2008) 1425
  -- 1441.
\newblock \href {http://dx.doi.org/http://dx.doi.org/10.1016/j.laa.2008.03.009}
  {\path{doi:http://dx.doi.org/10.1016/j.laa.2008.03.009}}.
\newline\urlprefix\url{http://www.sciencedirect.com/science/article/pii/S0024379508001456}

\bibitem{doi:10.1137/050637601}
R.~Byers, D.~Kressner, \href{http://dx.doi.org/10.1137/050637601}{Structured
  condition numbers for invariant subspaces}, SIAM Journal on Matrix Analysis
  and Applications 28~(2) (2006) 326--347.
\newblock \href {http://arxiv.org/abs/http://dx.doi.org/10.1137/050637601}
  {\path{arXiv:http://dx.doi.org/10.1137/050637601}}, \href
  {http://dx.doi.org/10.1137/050637601} {\path{doi:10.1137/050637601}}.
\newline\urlprefix\url{http://dx.doi.org/10.1137/050637601}

\bibitem{Ipsen04}
C.~F. Ipsen, Accurate eigenvalues for fast trains, SIAM News 37~(2004).

\bibitem{Hilliges2004}
A.~Hilliges, C.~Mehl, V.~Mehrmann, On the solution of palindromic eigenvalue
  problems, in: Proceedings of the 4th European Congress on Computational
  Methods in Applied Sciences and Engineering (ECCOMAS), Jyv\"{a}skyl\"{a},
  Finland, 2004.

\bibitem{Chu2010}
E.~K.-w. Chu, T.-M. Huang, W.-W. Lin, C.-T. Wu, Palindromic eigenvalue
  problems: A brief survey, Taiwanese Journal of Mathematics 14~(3A) (2010) 743--779.

\bibitem{za06}
S.~{Z}aglmayr, Eigenvalue Problems in SAW-Filter Simulations, Diplomarbeit,
  Institute of Computational Mathematics, Johannes Kepler University Linz,
  2002.

\bibitem{Zhou2011}
B.~Zhou, J.~Lam, G.-R. Duan,
  \href{http://dx.doi.org/10.1016/j.laa.2011.03.003}{Toward solution of matrix
  equation {$X=Af(X)B+C$}}, Linear Algebra and its Applications 435~(6) (2011) 1370--1398.
\newblock \href {http://dx.doi.org/10.1016/j.laa.2011.03.003}
  {\path{doi:10.1016/j.laa.2011.03.003}}.
\newline\urlprefix\url{http://dx.doi.org/10.1016/j.laa.2011.03.003}

\bibitem{Li2014546}
Z.-Y. Li, B.~Zhou, J.~Lam,
  \href{http://www.sciencedirect.com/science/article/pii/S0096300314005062}{Towards
  positive definite solutions of a class of nonlinear matrix equations},
  Applied Mathematics and Computation 237~(0) (2014) 546 -- 559.
\newblock \href {http://dx.doi.org/http://dx.doi.org/10.1016/j.amc.2014.03.135}
  {\path{doi:http://dx.doi.org/10.1016/j.amc.2014.03.135}}.
\newline\urlprefix\url{http://www.sciencedirect.com/science/article/pii/S0096300314005062}

\bibitem{Chiang2014925}
C.-Y. Chiang,
  \href{http://www.sciencedirect.com/science/article/pii/S0096300314010182}{On
  the solution of the linear matrix equation {$X=Af(X)B+C$}}, Applied
  Mathematics and Computation 244~(0) (2014) 925 -- 935.
\newblock \href {http://dx.doi.org/http://dx.doi.org/10.1016/j.amc.2014.07.061}
  {\path{doi:http://dx.doi.org/10.1016/j.amc.2014.07.061}}.
\newline\urlprefix\url{http://www.sciencedirect.com/science/article/pii/S0096300314010182}

\bibitem{Chu1987}
K.-w.~E. Chu, \href{http://dx.doi.org/10.1016/S0024-3795(87)90314-4}{The
  solution of the matrix equations {$AXB-CXD=E$} and {$(YA-DZ,YC-BZ)=(E,F)$}},
  Linear Algebra and its Applications 93 (1987) 93--105.
\newblock \href {http://dx.doi.org/10.1016/S0024-3795(87)90314-4}
  {\path{doi:10.1016/S0024-3795(87)90314-4}}.
\newline\urlprefix\url{http://dx.doi.org/10.1016/S0024-3795(87)90314-4}

\bibitem{Herndez1989333}
V.~Hern$\acute{a}$ndez, M.~Gass$\acute{a}$,
  \href{http://www.sciencedirect.com/science/article/pii/0024379589907088}{Explicit
  solution of the matrix equation {$AXB-CXD=E$}}, Linear Algebra and its
  Applications 121~(0) (1989) 333 -- 344.
\newblock \href
  {http://dx.doi.org/http://dx.doi.org/10.1016/0024-3795(89)90708-8}
  {\path{doi:http://dx.doi.org/10.1016/0024-3795(89)90708-8}}.
\newline\urlprefix\url{http://www.sciencedirect.com/science/article/pii/0024379589907088}

\bibitem{Zhou09}
B.~Zhou, G.-R. Duan,
  \href{http://imamci.oxfordjournals.org/content/26/1/59.abstract}{Parametric
  solutions to the generalized discrete {S}ylvester matrix equation {$MXN-X=TY$}
  and their applications}, IMA Journal of Mathematical Control and Information
  26~(1) (2009) 59--78.
\newblock \href
  %{http://arxiv.org/abs/http://imamci.oxfordjournals.org/content/26/1/59.full.pdf+html}
  %{\path{arXiv:http://imamci.oxfordjournals.org/content/26/1/59.full.pdf+html}},
  \href {http://dx.doi.org/10.1093/imamci/dnn013}
  %{\path{doi:10.1093/imamci/dnn013}}.
\newline\urlprefix\url{http://imamci.oxfordjournals.org/content/26/1/59.abstract}

\bibitem{Jameson1968}
A.~Jameson, \href{http://www.jstor.org/stable/2099227}{Solution of the equation
  ${AX + XB = C}$ by inversion of an ${M \times M}$ or ${N \times N}$ matrix},
  SIAM Journal on Applied Mathematics 16~(5) (1968) 1020--1023.
\newline\urlprefix\url{http://www.jstor.org/stable/2099227}

\bibitem{6595170}
A.-G. Wu, \href{http://dx.doi.org/10.1016/10.1049/iet-cta.2013.0075}{Explicit
  solutions to the matrix equation ${E\overline{X}F - AX = C}$}, IET Control Theory and
  Applications 7~(12) (2013) 1589--1598.
\newblock \href {http://dx.doi.org/10.1049/iet-cta.2013.0075}
  {\path{doi:10.1049/iet-cta.2013.0075}}.
\newline\urlprefix\url{http://dx.doi.org/10.1016/10.1049/iet-cta.2013.0075}

\bibitem{Langenhop1971329}
C.~Langenhop,
  \href{http://www.sciencedirect.com/science/article/pii/0024379571900048}{The
  {L}aurent expansion for a nearly singular matrix}, Linear Algebra and its
  Applications 4~(4) (1971) 329 -- 340.
\newblock \href
  {http://dx.doi.org/http://dx.doi.org/10.1016/0024-3795(71)90004-8}
  {\path{doi:http://dx.doi.org/10.1016/0024-3795(71)90004-8}}.
\newline\urlprefix\url{http://www.sciencedirect.com/science/article/pii/0024379571900048}

\bibitem{Schweitzer1993237}
P.~Schweitzer, G.~Stewart,
  \href{http://www.sciencedirect.com/science/article/pii/002437959390435Q}{The
  {L}aurent expansion of pencils that are singular at the origin}, Linear
  Algebra and its Applications 183~(0) (1993) 237 -- 254.
\newblock \href
  {http://dx.doi.org/http://dx.doi.org/10.1016/0024-3795(93)90435-Q}
  {\path{doi:http://dx.doi.org/10.1016/0024-3795(93)90435-Q}}.
\newline\urlprefix\url{http://www.sciencedirect.com/science/article/pii/002437959390435Q}

\bibitem{Rothblum198133}
U.~G. Rothblum,
  \href{http://www.sciencedirect.com/science/article/pii/0024379581900069}{Resolvent
  expansions of matrices and applications}, Linear Algebra and its Applications
  38~(0) (1981) 33 -- 49.
\newblock \href
  {http://dx.doi.org/http://dx.doi.org/10.1016/0024-3795(81)90006-9}
  {\path{doi:http://dx.doi.org/10.1016/0024-3795(81)90006-9}}.
\newline\urlprefix\url{http://www.sciencedirect.com/science/article/pii/0024379581900069}

\bibitem{4047765}
F.~Lewis, Adjoint matrix, {B}$\acute{\mbox{e}}$zout theorem, {C}ayley-{H}amilton
  theorem, and {F}adeev's method for the matrix pencil $sE-A$, 1983. The 22nd IEEE Conference 1282--1288.
\newblock \href {http://dx.doi.org/10.1109/CDC.1983.269734}
  {\path{doi:10.1109/CDC.1983.269734}}.

\bibitem{Chiang2012}
C.-Y. Chiang, E.~K.-W. Chu, W.-W. Lin,
  \href{http://dx.doi.org/10.1016/j.amc.2012.01.065}{On the
  {$\star$}-{S}ylvester equation {$AX\pm X^\star B^\star=C$}}, Applied Mathematics and Computation 218~(17) (2012) 8393--8407.
\newblock \href {http://dx.doi.org/10.1016/j.amc.2012.01.065}
  {\path{doi:10.1016/j.amc.2012.01.065}}.
\newline\urlprefix\url{http://dx.doi.org/10.1016/j.amc.2012.01.065}

\bibitem{Chiang2013AAA}
C.-Y. Chiang, \href{http://www.hindawi.com/journals/aaa/2013/824641/}{A note on
  the $\top$-{S}tein matrix equation}, Abstract and Applied Analysis 2013~(Article ID
  824641), 8 pages.
\newline\urlprefix\url{http://www.hindawi.com/journals/aaa/2013/824641/}

\bibitem{Zhou2009327}
B.~Zhou, Z.-Y. Li, G.-R. Duan, Y.~Wang,
  \href{http://www.sciencedirect.com/science/article/pii/S0096300309001489}{Solutions
  to a family of matrix equations by using the {K}ronecker matrix polynomials},
  Applied Mathematics and Computation 212~(2) (2009) 327 -- 336.
\newblock \href {http://dx.doi.org/http://dx.doi.org/10.1016/j.amc.2009.02.021}
  {\path{doi:http://dx.doi.org/10.1016/j.amc.2009.02.021}}.
\newline\urlprefix\url{http://www.sciencedirect.com/science/article/pii/S0096300309001489}

\bibitem{1979}
J.~R. Magnus, H.~Neudecker, \href{http://www.jstor.org/stable/2958818}{The
  commutation matrix: Some properties and applications}, The Annals of
  Statistics 7~(2) (1979) pp. 381--394.
\newline\urlprefix\url{http://www.jstor.org/stable/2958818}

\end{thebibliography}

\end{document}